\documentclass[twocolumn]{autart}    

\usepackage{amsmath} 
\usepackage{amssymb}  
\newenvironment{proof}{\emph{Proof.}}{\hfill $\square$ \\}

\begin{document}

\begin{frontmatter}

\title{Differential-Geometric Decomposition of\\Flat Nonlinear Discrete-Time Systems\thanksref{footnoteinfo}} 

\thanks[footnoteinfo]{The material in this paper was partially presented at the 10th IFAC
Symposium on Nonlinear Control Systems (NOLCOS 2016), August 23--25, 2016, Monterey, California, USA.}

\author{Bernd Kolar\corauthref{cor}}\ead{bernd\underline{\ }kolar@ifac-mail.org},    
\author{Markus Sch\"{o}berl}\ead{markus.schoeberl@jku.at},  
\author{Johannes Diwold}\ead{johannes.diwold@jku.at}              

\corauth[cor]{Corresponding author.}


\address{Institute of Automatic Control and Control Systems Technology, Johannes Kepler University, Linz, Austria}


\begin{abstract}                          
We prove that every flat nonlinear discrete-time system can be decomposed by coordinate transformations into a smaller-dimensional subsystem and an endogenous dynamic feedback. For flat continuous-time systems, no comparable result is available. The advantage of such a decomposition is that the complete system is flat if and only if the subsystem is flat.
Thus, by repeating the decomposition at most $n-1$ times, where $n$ is the dimension of the state space, the flatness of a discrete-time system can be checked in an algorithmic way. If the system is flat, then the algorithm yields a flat output which only depends on the state variables. Hence, every flat discrete-time system has a flat output which does not depend on the inputs and their forward-shifts. Again, no comparable result for flat continuous-time systems is available. The algorithm requires in each decomposition step the construction of state- and input transformations, which are obtained by straightening out certain vector fields or distributions with the flow-box theorem or the Frobenius theorem. Thus, from a computational point of view, only the calculation of flows and the solution of algebraic equations is needed. We illustrate our results by two examples.
\end{abstract}

\end{frontmatter}

\section{Introduction}

The concept of flatness has been introduced by Fliess, L{\'e}vine,
Martin and Rouchon in the 1990s for nonlinear continuous-time systems
(see e.g. \cite{FliessLevineMartinRouchon:1992}, \cite{FliessLevineMartinRouchon:1995},
and \cite{FliessLevineMartinRouchon:1999}). Flat continuous-time
systems have the characteristic feature that all system variables
can be expressed by a flat output and its time derivatives. They form
an extension of the class of static feedback linearizable systems,
and can be linearized by an endogenous dynamic feedback. The reason
for the ongoing popularity of flat systems lies in the fact that the
knowledge of a flat output allows an elegant systematic solution to
motion planning problems as well as the design of tracking controllers.
However, in contrast to the static feedback linearization problem,
which has been solved in \cite{JakubczykRespondek:1980} and \cite{HuntSu:1981},
there still exist no efficiently verifiable necessary and sufficient
conditions for flatness, and the construction of flat outputs is a
challenging problem.

For nonlinear discrete-time systems, flatness can be defined analogously
to the continuous-time case. The main difference is that time derivatives
are replaced by forward-shifts. To distinguish both concepts, often
the terms differential flatness and difference flatness are used (see
e.g. \cite{Sira-RamirezAgrawal:2004}). Like in the continuous-time
case, flat discrete-time systems form an extension of the class of
static feedback linearizable systems, and can be linearized by an
endogenous dynamic feedback (see e.g. \cite{KaldmaeKotta:2013}).
The static feedback linearization problem for discrete-time systems
has already been studied and solved in several papers using different
mathematical frameworks, see \cite{Grizzle:1986}, \cite{Jakubczyk:1987},
and \cite{Aranda-BricaireKottaMoog:1996}. There exist verifiable
necessary and sufficient conditions, which give rise to an algorithm
for the calculation of a linearizing output. The more general dynamic
feedback linearization problem, which includes flatness as a special
case, has been studied for discrete-time systems e.g. in \cite{Aranda-BricaireKottaMoog:1996}
and \cite{Aranda-BricaireMoog:2008}. In particular,
\cite{Aranda-BricaireMoog:2008} addresses the difference between
linearization by endogenous and exogenous dynamic feedback for discrete-time
systems. However, like in the continuous-time case, no efficiently
verifiable necessary and sufficient conditions are available. Thus,
the construction of flat outputs is also a difficult problem.

In practical applications, flat outputs often have some physical meaning,
see e.g. \cite{FliessLevineMartinRouchon:1999}. Therefore, the construction
of flat outputs is -- like the construction of Lyapunov functions
-- often based on physical considerations. A possible more systematic
approach is to transform the system into a decomposed form, where
the complete system is flat if and only if a smaller-dimensional subsystem
is flat. Repeating this decomposition with the subsystem may then
lead after several steps to a flat output. Such methods have been
developed with different types of decompositions for continuous-time
systems in \cite{SchlacherSchoberl:2007}, \cite{SchlacherSchoberl:2013},
\cite{Schoberl:2014}, and \cite{SchoberlSchlacher:2014}, and they
were transferred to discrete-time systems in \cite{KolarKaldmaeSchoberlKottaSchlacher:2016}
and \cite{KolarSchoberlSchlacher:2016-2} (see also \cite{Kolar:2017}).
The fundamental question is, however, under which conditions such
decompositions exist, and whether every flat system allows a decomposition
or not. For continuous-time systems, this question is a very difficult
one. For discrete-time systems, in contrast, the situation is completely
different. We present a simple geometric proof that a flat discrete-time
system can always be transformed by state- and input transformations
into a subsystem and an endogenous dynamic feedback. This type of
decomposition has been studied in \cite{KolarKaldmaeSchoberlKottaSchlacher:2016}
both in a differential-geometric and an algebraic framework, but without
a proof that for flat systems the decomposition is always possible.
In the present paper, we focus on the geometric framework. For a further
discussion in the algebraic framework, see \cite{Kaldmae:2016}. The
advantage of the geometric approach is that the decompositions can
be constructed systematically in special coordinates, and that the
proof for the existence of a decomposition of flat systems becomes
particularly simple. As a consequence of the latter result, the flatness
of discrete-time systems can be checked in an algorithmic way. If
the system is flat, then a repeated decomposition will yield a flat
output after at most $n-1$ steps, where $n$ denotes the dimension
of the state space. Since the constructed flat output
only depends on the state variables, we obtain the additional result
that every flat discrete-time system has a flat output which is independent
of the inputs and their forward-shifts.

The paper is organized as follows: In Section \ref{sec:DiscreteTimeSystems_and_Flatness}
we recall the definition of difference flatness and give an overview
of some important properties of flat discrete-time systems. In Section
\ref{sec:Decomposition} we discuss the decomposition of discrete-time
systems into a subsystem and an endogenous dynamic feedback by means
of coordinate transformations. We give geometric conditions for the
existence of such a decomposition, and show that for flat systems
these conditions are always satisfied. In Section \ref{sec:Algorithm}
we present an algorithm for the calculation of flat outputs, which
is based on a repeated application of the decomposition of Section
\ref{sec:Decomposition}. Furthermore, we show that
every flat discrete-time system has a flat output which only
depends on the state variables. We illustrate our results by two
examples in Section \ref{sec:Examples}.

\section{\label{sec:DiscreteTimeSystems_and_Flatness}Discrete-Time Systems
and Flatness}

In this contribution we consider discrete-time systems
\begin{equation}
x^{i,+}=f^{i}(x,u)\,,\quad i=1,\ldots,n\label{eq:sys}
\end{equation}
in state representation with $\dim(x)=n$, $\dim(u)=m$, and smooth
functions $f^{i}(x,u)$. Geometrically, such a system can be interpreted
as a map $f$ from a manifold $\mathcal{X}\times\mathcal{U}$ with
coordinates $(x,u)$ to a manifold $\mathcal{X}\text{\textsuperscript{+}}$
with coordinates $x^{+}$. We assume throughout the paper that the
system meets
\[
\mathrm{rank}(\partial_{(x,u)}f)=n\,,
\]
which means that the map $f$ is a submersion and therefore locally
surjective. Since this assumption is necessary for accessibility (see
e.g. \cite{Grizzle:1993}) and consequently also for flatness, it
is no restriction. To achieve the desired decompositions, we will
use state- and input transformations
\begin{equation}
\begin{array}{ccll}
\bar{x}^{i} & = & \Phi_{x}^{i}(x)\,,\quad & i=1,\ldots,n\\
\bar{u}^{j} & = & \Phi_{u}^{j}(x,u)\,, & j=1,\ldots,m\,,
\end{array}\label{eq:state_and_input_trans}
\end{equation}
and it should be noted that the variables $x^{+}$ are transformed
of course in the same way as the variables $x$. The transformed system
is given by
\[
\bar{x}^{i,+}=\underbrace{\Phi_{x}^{i}(x^{+})\circ f(x,u)\circ\hat{\Phi}(\bar{x},\bar{u})}_{\bar{f}(\bar{x},\bar{u})}\,,\quad i=1,\ldots,n\,,
\]
with the inverse $(x,u)=\hat{\Phi}(\bar{x},\bar{u})$ of (\ref{eq:state_and_input_trans}).
The superscript $+$ is only used to denote the forward-shift of the
state variables $x$. For the inputs and flat outputs we also need
higher forward-shifts, and use a subscript in brackets instead. For
instance, $u_{[\alpha]}$ denotes the $\alpha$-th forward-shift of
$u$. To keep formulas short and readable, we also use the Einstein
summation convention. Furthermore, we want to emphasize that all our
results are local. This is due to the use of the inverse- and the
implicit function theorem, the flow-box theorem, and the Frobenius
theorem, which allow only local results. We also assume that all functions
are smooth in order to avoid mathematical subtleties.

In the following, we summarize the concept of difference flatness,
which is the discrete-time counterpart of differential flatness for
continuous-time systems. Roughly speaking, the main difference is
that time derivatives are replaced by forward-shifts. Since many results
can be shown in a similar way to the continuous-time case, we omit
detailed proofs. Analogously to the static feedback
linearization problem for discrete-time systems, we define flatness
around an equilibrium
\begin{equation}
x_{0}^{i}=f^{i}(x_{0},u_{0})\,,\quad i=1,\ldots,n\label{eq:equilibrium}
\end{equation}
of the system (\ref{eq:sys}). The reason is that even in one time
step the state of a discrete-time system can move far away from the
initial state, regardless of the input values. Thus, in order not
to loose localness, we consider a suitable neighborhood of an equilibrium.
To introduce the concept of difference flatness, we need a space with
coordinates $(x,u,u_{[1]},u_{[2]},\ldots)$. On this space we have
the forward-shift operator $\delta_{xu}$, which acts on a function
$g$ according to the rule
\[
\delta_{xu}(g(x,u,u_{[1]},u_{[2]},\ldots))=g(f(x,u),u_{[1]},u_{[2]},u_{[3]},\ldots)\,.
\]
A repeated application of $\delta_{xu}$ is denoted by $\delta_{xu}^{\alpha}$.
In this framework, an equilibrium (\ref{eq:equilibrium})
corresponds to a point $(x_{0},u_{0},u_{0},u_{0},\ldots)$, and flatness
of discrete-time systems can be defined as follows.
\begin{defn}
The system (\ref{eq:sys}) is said to be flat around
an equilibrium $(x_{0},u_{0})$, if the $n+m$ coordinate functions
$x$ and $u$ can be expressed locally by an $m$-tuple
of functions
\begin{equation}
y^{j}=\varphi^{j}(x,u,u_{[1]},\ldots,u_{[q]})\,,\quad j=1,\ldots,m\label{eq:flat_output}
\end{equation}
and their forward-shifts
\[
\begin{array}{ccl}
y_{[1]} & = & \delta_{xu}(\varphi(x,u,u_{[1]},\ldots,u_{[q]}))\\
y_{[2]} & = & \delta_{xu}^{2}(\varphi(x,u,u_{[1]},\ldots,u_{[q]}))\\
 & \vdots
\end{array}
\]
up to some finite order. The $m$-tuple (\ref{eq:flat_output}) is
called a flat output.
\end{defn}

If (\ref{eq:flat_output}) is a flat output, then the $m(\beta+1)$
functions $\varphi,\delta_{xu}(\varphi),\delta_{xu}^{2}(\varphi),\ldots,\delta_{xu}^{\beta}(\varphi)$
are functionally independent for arbitrary $\beta\geq0$.\footnote{We only sketch the proof of this statement: If $u$ can be expressed
by the flat output and its forward-shifts, then this is also possible
for all forward-shifts $u_{[\alpha]}$ of $u$. By using the facts
that the coordinate functions $u,u_{[1]},u_{[2]},\ldots$ are functionally
independent and $\dim(y)=\dim(u)=m$, it can be shown that the functions
$\varphi,\delta_{xu}(\varphi),\delta_{xu}^{2}(\varphi),\ldots,\delta_{xu}^{\beta}(\varphi)$
must also be functionally independent.} Therefore, the representation of $x$ and $u$ by the flat output
and its forward-shifts is unique, and it has the form
\begin{equation}
\begin{array}{cclcl}
x^{i} & = & F_{x}^{i}(y_{[0,R-1]})\,, & \quad & i=1,\ldots,n\\
u^{j} & = & F_{u}^{j}(y_{[0,R]})\,, & \quad & j=1,\ldots,m\,.
\end{array}\label{eq:flat_parametrization}
\end{equation}
The multi-index $R=(r_{1},\ldots,r_{m})$ contains the number of forward-shifts
of each component of the flat output which is needed to express $x$
and $u$, and $y_{[0,R]}$ is an abbreviation for $y$ and its forward-shifts
up to order $R$. Written in components,
\[
y_{[0,R]}=(y_{[0,r_{1}]}^{1},\ldots,y_{[0,r_{m}]}^{m})
\]
with
\[
y_{[0,r_{j}]}^{j}=(y^{j},y_{[1]}^{j},\ldots,y_{[r_{j}]}^{j})\,,\quad j=1,\ldots,m\,.
\]
With the forward-shift operator $\delta_{y}$ in coordinates $(y,y_{[1]},y_{[2]},\ldots)$,
which acts on a function $h$ according to the rule
\begin{equation}
\delta_{y}(h(y,y_{[1]},y_{[2]},\ldots))=h(y_{[1]},y_{[2]},y_{[3]},\ldots)\,,\label{eq:delta_y}
\end{equation}
the parametrization of arbitrary forward-shifts $u_{[\alpha]}$ of
$u$ follows from (\ref{eq:flat_parametrization}) as
\[
u_{[\alpha]}^{j}=\delta_{y}^{\alpha}(F_{u}^{j}(y_{[0,R]}))\,,\quad j=1,\ldots,m.
\]
It is a well-known fact that the parametrization $F_{x}$ of the state
only depends on $y_{[0,R-1]}$, and that the highest forward-shifts
$y_{[R]}=(y_{[r_{1}]}^{1},\ldots,y_{[r_{m}]}^{m})$ that are required
in (\ref{eq:flat_parametrization}) only appear in the parametrization
$F_{u}$ of the input. It is also not hard to show that the map $(x,u)=F(y_{[0,R]})$
given by (\ref{eq:flat_parametrization}) is a submersion, i.e., that
the rows of its Jacobian matrix are linearly independent. Likewise,
the map
\begin{equation}
\begin{array}{ccl}
y & = & \varphi(x,u,u_{[1]},\ldots,u_{[q]})\\
y_{[1]} & = & \delta_{xu}(\varphi(x,u,u_{[1]},\ldots,u_{[q]}))\\
y_{[2]} & = & \delta_{xu}^{2}(\varphi(x,u,u_{[1]},\ldots,u_{[q]}))\\
 & \vdots\\
y_{[R]} & = & \delta_{xu}^{R}(\varphi(x,u,u_{[1]},\ldots,u_{[q]}))
\end{array}\label{eq:flat_parametrization_inverse}
\end{equation}
is also a submersion. This is a simple consequence of the already
mentioned functional independence of the flat output and its forward-shifts.
If the system (\ref{eq:sys}) is static feedback linearizable and
$y=\varphi(x)$ is a linearizing output, then the submersion (\ref{eq:flat_parametrization})
becomes a diffeomorphism, and its inverse is given by (\ref{eq:flat_parametrization_inverse}).
In this case, the parametrization (\ref{eq:flat_parametrization})
can be used as a coordinate transformation which transforms the system
(\ref{eq:sys}) into the discrete-time Brunovsky normal form.

If we substitute the parametrization (\ref{eq:flat_parametrization})
into the identity
\[
\delta_{xu}(x^{i})=f^{i}(x,u)\,,\quad i=1,\ldots,n\,,
\]
we get the important identity
\begin{equation}
\delta_{y}(F_{x}^{i}(y_{[0,R-1]}))=f^{i}\circ F(y_{[0,R]})\,,\quad i=1,\ldots,n\,.\label{eq:sys_identity_y}
\end{equation}
Because of (\ref{eq:sys_identity_y}), it is obvious that $F_{x}$
can indeed only depend on $y_{[0,R-1]}$. Otherwise, $\delta_{y}(F_{x})$
would depend on forward-shifts of $y$ that are not contained in $y_{[0,R]}$.
A further fundamental consequence of the identity (\ref{eq:sys_identity_y})
and the special form of the forward-shift operator (\ref{eq:delta_y})
is that the system equations (\ref{eq:sys}) do not impose any restrictions
on the feasible trajectories
\begin{equation}
y^{j}(k)\,,\quad j=1,\ldots,m\label{eq:trajectory_y}
\end{equation}
of the flat output (\ref{eq:flat_output}). That is, for every trajectory
(\ref{eq:trajectory_y}) of the flat output there exists a uniquely
determined solution $(x(k),u(k))$ of the system (\ref{eq:sys}) such
that the equations
\[
y^{j}(k)=\varphi^{j}(x(k),u(k),u(k+1),\ldots,u(k+q))\,,
\]
$j=1,\ldots,m$ are satisfied identically. The trajectories $x(k)$
and $u(k)$ of state and input are determined by $y(k)$ and its forward-shifts
via the parametrization (\ref{eq:flat_parametrization}). Thus, just
like in the case of differentially flat continuous-time systems, there
is a one-to-one correspondence between solutions of the system (\ref{eq:sys})
and arbitrary trajectories of the flat output.

\section{\label{sec:Decomposition}Decomposition of Flat Systems}

In this section we deal with a transformation of the system (\ref{eq:sys})
into a certain decomposed form, which can be interpreted as a splitting
into a subsystem and an endogenous dynamic feedback. This decomposed
form has the property that the complete system is flat if and only
if the subsystem is flat.
\begin{lem}
\label{lem:basic_decomposition_flat}A system of the form
\begin{equation}
\begin{array}{lcl}
x_{1}^{i_{1},+}=f_{1}^{i_{1}}(x_{1},x_{2},u_{1})\,, & \quad & i_{1}=1,\ldots,n-m_{2}\\
x_{2}^{i_{2},+}=f_{2}^{i_{2}}(x_{1},x_{2},u_{1},u_{2})\,, &  & i_{2}=1,\ldots,m_{2}
\end{array}\label{eq:basic_decomposition_flat}
\end{equation}
with $\dim(u_{2})=\dim(x_{2})=m_{2}$ and $\mathrm{rank}(\partial_{u}f)=\dim(u)=m$
is flat if and only if the subsystem
\begin{equation}
x_{1}^{+}=f_{1}(x_{1},x_{2},u_{1})\label{eq:basic_decomposition_flat_subsys}
\end{equation}
with the $m$ inputs $(x_{2},u_{1})$ is flat.
\end{lem}

\begin{proof}
\emph{\ }\\
\emph{Flatness of (\ref{eq:basic_decomposition_flat_subsys}) $\Rightarrow$
Flatness of (\ref{eq:basic_decomposition_flat}):} If $y$ is a flat
output of the subsystem (\ref{eq:basic_decomposition_flat_subsys}),
then the system variables $x_{1}$, $x_{2}$, and $u_{1}$ of this
subsystem can be expressed as functions of $y$ and its forward-shifts.
Because of the regularity of the Jacobian matrix $\partial_{u_{2}}f_{2}$,
which is an immediate consequence of $\mathrm{rank}(\partial_{u}f)=\dim(u)$
and the structure of (\ref{eq:basic_decomposition_flat}), the implicit
function theorem allows to express $u_{2}$ as function of $x_{1}$,
$x_{2}$, $u_{1}$, and $x_{2}^{+}$. Consequently, $u_{2}$ can also
be expressed as a function of $y$ and its forward-shifts, and $y$
is a flat output of the complete system (\ref{eq:basic_decomposition_flat}).\emph{}\\
\emph{Flatness of (\ref{eq:basic_decomposition_flat}) $\Rightarrow$
Flatness of (\ref{eq:basic_decomposition_flat_subsys}):} Because
of the regularity of $\partial_{u_{2}}f_{2}$, we can perform an input
transformation
\[
\hat{u}_{2}^{j_{2}}=f_{2}^{j_{2}}(x_{1},x_{2},u_{1},u_{2})\,,\quad j_{2}=1,\ldots,m_{2}
\]
such that (\ref{eq:basic_decomposition_flat}) takes the simpler form
\begin{equation}
\begin{array}{lcl}
x_{1}^{i_{1},+}=f_{1}^{i_{1}}(x_{1},x_{2},u_{1})\,, & \quad & i_{1}=1,\ldots,n-m_{2}\\
x_{2}^{i_{2},+}=\hat{u}_{2}^{i_{2}}\,, &  & i_{2}=1,\ldots,m_{2}\,.
\end{array}\label{eq:basic_decomposition_flat_normed}
\end{equation}
If
\[
y=\varphi(x_{1},x_{2},u_{1},\hat{u}_{2},u_{1,[1]},\hat{u}_{2,[1]},\ldots,u_{1,[q]},\hat{u}_{2,[q]})
\]
is a flat output of (\ref{eq:basic_decomposition_flat_normed}), then
by substituting $\hat{u}_{2}^{j_{2}}=x_{2,[1]}^{j_{2}}$ and $\hat{u}_{2,[\alpha]}^{j_{2}}=x_{2,[\alpha+1]}^{j_{2}}\,,\,\alpha\geq1$
we immediately get a flat output of the subsystem (\ref{eq:basic_decomposition_flat_subsys}).
\end{proof}
Note that the Jacobian matrix $\partial_{(x_{2},u_{1})}f_{1}$
does not necessarily have rank $m$. Thus, the subsystem (\ref{eq:basic_decomposition_flat_subsys})
may have redundant inputs. In this case, a flat output of the subsystem
(\ref{eq:basic_decomposition_flat_subsys}) contains components of
$x_{2}$.
\begin{rem}
\label{rem:redundant_inputs_state_variables}The
structure of (\ref{eq:basic_decomposition_flat}) and
\[
\mathrm{rank}\left(\left[\begin{array}{cc}
\partial_{u_{1}}f_{1} & 0\\
\partial_{u_{1}}f_{2} & \partial_{u_{2}}f_{2}
\end{array}\right]\right)=\mathrm{rank}(\partial_{u}f)=m
\]
imply $\mathrm{rank}(\partial_{u_{1}}f_{1})=m-m_{2}=\dim(u_{1})$.
As a consequence, the redundant inputs of the subsystem (\ref{eq:basic_decomposition_flat_subsys})
can always be found among the variables $x_{2}$.
\end{rem}

The equations
\[
x_{2}^{i_{2},+}=f_{2}^{i_{2}}(x_{1},x_{2},u_{1},u_{2})\,,\quad i_{2}=1,\ldots,m_{2}
\]
of (\ref{eq:basic_decomposition_flat}) can be interpreted as an endogenous
dynamic feedback for the subsystem (\ref{eq:basic_decomposition_flat_subsys}).
This is in accordance with the fact that applying or removing an endogenous
dynamic feedback has no effect on the flatness of a system.

Our next objective is to derive necessary and sufficient differential-geometric
conditions for the existence of a transformation of the system (\ref{eq:sys})
into the decomposed form (\ref{eq:basic_decomposition_flat}). To
formulate these conditions, we use the notion of $f$-related vector
fields. For completeness, we briefly explain the basics. More details
can be found in \cite{Boothby:1986}. By $f_{*}:\mathcal{T}(\mathcal{X}\times\mathcal{U})\rightarrow\mathcal{T}(\mathcal{X}^{+})$
we denote the tangent map of $f:\mathcal{X}\times\mathcal{U}\rightarrow\mathcal{X}^{+}$,
and by $f_{*p}:\mathcal{T}_{p}(\mathcal{X}\times\mathcal{U})\rightarrow\mathcal{T}_{f(p)}(\mathcal{X}^{+})$
we denote the tangent map of $f$ at some point $p\in\mathcal{X}\times\mathcal{U}$.
If
\begin{equation}
v=v_{x}^{i}(x,u)\partial_{x^{i}}+v_{u}^{j}(x,u)\partial_{u^{j}}\label{eq:f-related_v}
\end{equation}
is a vector field on $\mathcal{X}\times\mathcal{U}$, then the vector
$f_{*p}(v_{p})$ at $f(p)\in\mathcal{X}^{+}$ is called the pushforward
of the vector $v_{p}$ at $p\in\mathcal{X}\times\mathcal{U}$ by $f$.
However, since $f$ is only a submersion and not a diffeomorphism,
the vector field $v$ does not necessarily induce a well-defined vector
field on $\mathcal{X}^{+}$. The problem is that the inverse image
$f^{-1}(q)$ of a point $q\in\mathcal{X}^{+}$ is an $m$-dimensional
submanifold of $\mathcal{X}\times\mathcal{U}$, and it may happen
that for a pair of points $p_{1}$ and $p_{2}$ on this submanifold
we get $f_{*p_{1}}(v_{p_{1}})\neq f_{*p_{2}}(v_{p_{2}})$. In other
words, the vector at the point $f(p_{1})=f(p_{2})=q$ may be not unique.
If, however, there exists a vector field
\begin{equation}
w=w^{i}(x^{+})\partial_{x^{i,+}}\label{eq:f-related_w}
\end{equation}
on $\mathcal{X}^{+}$ such that for all $q\in\mathcal{X}^{+}$ and
$p\in f^{-1}(q)\subset\mathcal{X}\times\mathcal{U}$ we have $f_{*p}(v_{p})=w_{q}$,
then the vector fields $v$ and $w$ are said to be $f$-related and
we write $w=f_{*}(v)$. In components, $f$-relatedness means
\[
w^{i}(x^{+})\circ f(x,u)=\partial_{x^{k}}f^{i}v_{x}^{k}(x,u)+\partial_{u^{j}}f^{i}v_{u}^{j}(x,u)\,,
\]
$i=1,\ldots,n$. Since we assume that $f$ is a submersion and therefore
locally surjective, the vector field (\ref{eq:f-related_w}) determined
by a given vector field (\ref{eq:f-related_v}) is unique if it exists.
Moreover, as a submersion, the map $f$ induces a fibration (foliation)
of the manifold $\mathcal{X}\times\mathcal{U}$ with $m$-dimensional
fibres (leaves). Thus, we will adopt some terminology used for fibre
bundles (see e.g. \cite{Saunders:1989}), and call vector fields (\ref{eq:f-related_v})
on $\mathcal{X}\times\mathcal{U}$ that are $f$-related to a vector
field (\ref{eq:f-related_w}) on $\mathcal{X}^{+}$ ``projectable''.
Similarly, we will call a distribution $D$ on $\mathcal{X}\times\mathcal{U}$
``projectable'' if it admits a basis that consists of projectable
vector fields. Since we deal particularly with involutive distributions,
we will also make use of the fact that the Lie brackets $[v_{1},v_{2}]$
and $[w_{1},w_{2}]$ of two pairs $v_{1},w_{1}$ and $v_{2},w_{2}$
of $f$-related vector fields are again $f$-related, i.e.,
\[
f_{*}[v_{1},v_{2}]=[w_{1},w_{2}]\,.
\]
For this reason, the pushforward $f_{*}D$ of an involutive projectable
distribution is again an involutive distribution.

Checking whether a vector field or distribution is projectable or
not becomes very simple if we use coordinates on $\mathcal{X}\times\mathcal{U}$
that are adapted to the fibration. Adapted coordinates can be introduced
by a transformation of the form
\begin{equation}
\begin{array}{ccl}
x^{i,+} & = & f^{i}(x,u)\,,\quad i=1,\ldots,n\\
\xi^{j} & = & h^{j}(x,u)\,,\quad j=1,\ldots,m\,,
\end{array}\label{eq:adapted_coordinates}
\end{equation}
where the $m$ functions $h^{j}(x,u)$ must be chosen in such a way
that (\ref{eq:adapted_coordinates}) is a (local) diffeomorphism.
Thus, the Jacobian matrix
\[
\left[\begin{array}{cc}
\partial_{x}f & \partial_{u}f\\
\partial_{x}h & \partial_{u}h
\end{array}\right]
\]
must be regular. Because of the linear independence of the rows of
the Jacobian matrix of a submersion, this is always possible. With
coordinates $(x^{+},\xi)$ on $\mathcal{X}\times\mathcal{U}$, the
map $f$ takes the simple form $f=\mathrm{pr}_{1}$. All points of
$\mathcal{X}\times\mathcal{U}$ with the same value of $x^{+}$ belong
to the same fibre and are mapped to the same point of $\mathcal{X}^{+}$,
regardless of the value of the fibre coordinates $\xi$. The vector
field (\ref{eq:f-related_v}) in adapted coordinates has in general
the form
\begin{equation}
v=a^{i}(x^{+},\xi)\partial_{x^{i,+}}+b^{j}(x^{+},\xi)\partial_{\xi^{j}}\,,\label{eq:f-related_v_adapt}
\end{equation}
and because of $f=\mathrm{pr}_{1}$ an application of the tangent
map $f_{*}$ to (\ref{eq:f-related_v_adapt}) yields
\begin{equation}
f_{*}(v)=a^{i}(x^{+},\xi)\partial_{x^{i,+}}\,.\label{eq:f-related_fstar_v_adapt}
\end{equation}
Obviously, (\ref{eq:f-related_fstar_v_adapt}) is a well-defined vector
field on $\mathcal{X}^{+}$ if and only if the functions $a^{i}$
are independent of the coordinates $\xi$. In this case, (\ref{eq:f-related_fstar_v_adapt})
corresponds to the vector field (\ref{eq:f-related_w}).

With these mathematical preliminaries, we can formulate conditions
for the existence of a transformation of the system (\ref{eq:sys})
into the form (\ref{eq:basic_decomposition_flat}).
\begin{thm}
\label{thm:decomposition_conditions_distribution}Consider a system
(\ref{eq:sys}) with $\mathrm{rank}(\partial_{u}f)=m$. There exists
a coordinate transformation\begin{subequations}\label{eq:decomposition_coord_transformation}
\begin{align}
(\bar{x}_{1},\bar{x}_{2})=\:\: & \Phi_{x}(x)\label{eq:decompostion_state_transformation}\\
(\bar{u}_{1},\bar{u}_{2})=\:\: & \Phi_{u}(x,u)\label{eq:decompostion_input_transformation}
\end{align}
\end{subequations}with $\dim(\bar{u}_{2})=\dim(\bar{x}_{2})=m_{2}$
such that in transformed coordinates the system has the form
\begin{equation}
\begin{array}{ccl}
\bar{x}_{1}^{+} & = & \bar{f}_{1}(\bar{x}_{1},\bar{x}_{2},\bar{u}_{1})\\
\bar{x}_{2}^{+} & = & \bar{f}_{2}(\bar{x}_{1},\bar{x}_{2},\bar{u}_{1},\bar{u}_{2})
\end{array}\label{eq:sys_decomposed}
\end{equation}
if and only if on $\mathcal{X}\times\mathcal{U}$ there exists an
$m_{2}$-dimensional projectable and involutive subdistribution $D\subset\mathrm{span}\{\partial_{u}\}$.
\end{thm}

\begin{proof}
\emph{\ }\\
\emph{Sufficiency:} Since $D$ is involutive and $D\subset\mathrm{span}\{\partial_{u}\}$,
because of the Frobenius theorem there exists an input transformation
(\ref{eq:decompostion_input_transformation}) with $\dim(\bar{u}_{2})=m_{2}$
such that $D=\mathrm{span}\{\partial_{\bar{u}_{2}}\}$. Furthermore,
since $D\subset\mathrm{span}\{\partial_{u}\}$ is projectable and
the Jacobian matrix $\partial_{u}f$ has full rank, the pushforward
$f_{*}D$ is a well-defined $m_{2}$-dimensional involutive distribution
on $\mathcal{X}^{+}$. Thus, because of the Frobenius theorem there
exists a state transformation (\ref{eq:decompostion_state_transformation})
with $\dim(\bar{x}_{2})=m_{2}$ such that $f_{*}D=\mathrm{span}\{\partial_{\bar{x}_{2}^{+}}\}$.\footnote{Note again that state transformations are performed simultaneously
for $x$ and $x^{+}$.} In these coordinates, the transformed map $\bar{x}^{+}=\bar{f}(\bar{x},\bar{u})$
has the form (\ref{eq:sys_decomposed}). This can be seen as follows:
Let $\bar{f}_{1}$ and $\bar{f}_{2}$ denote the $\bar{x}_{1}$- and
$\bar{x}_{2}$-components of $\bar{f}$. Then the (pointwise defined)
pushforwards of the vector fields $\partial_{\bar{u}_{2}^{j_{2}}}$,
$j_{2}=1,\ldots,m_{2}$ are given by
\[
f_{*}(\partial_{\bar{u}_{2}^{j_{2}}})=\partial_{\bar{u}_{2}^{j_{2}}}\bar{f}_{1}^{i_{1}}\partial_{\bar{x}_{1}^{i_{1},+}}+\partial_{\bar{u}_{2}^{j_{2}}}\bar{f}_{2}^{i_{2}}\partial_{\bar{x}_{2}^{i_{2},+}}\,.
\]
Since by construction $f_{*}(\partial_{\bar{u}_{2}^{j_{2}}})\in f_{*}D=\mathrm{span}\{\partial_{\bar{x}_{2}^{+}}\}$,
we immediately get
\[
\partial_{\bar{u}_{2}^{j_{2}}}\bar{f}_{1}^{i_{1}}=0\,,\quad i_{1}=1,\ldots,n-m_{2},\,j_{2}=1,\ldots,m_{2}\,,
\]
which shows that the functions $\bar{f}_{1}^{i_{1}}$ are independent
of $\bar{u}_{2}$.\emph{}\\
\emph{Necessity:} To prove necessity, assume that there exists a coordinate
transformation (\ref{eq:decomposition_coord_transformation}) such
that (\ref{eq:sys}) takes the form (\ref{eq:sys_decomposed}). Because
of $\mathrm{rank}(\partial_{\bar{u}_{2}}\bar{f}_{2})=m_{2}$, there
exists a further input transformation $\hat{u}_{2}^{j_{2}}=\bar{f}_{2}^{j_{2}}(\bar{x}_{1},\bar{x}_{2},\bar{u}_{1},\bar{u}_{2})$,
$j_{2}=1,\ldots,m_{2}$ such that the system is of the form (\ref{eq:basic_decomposition_flat_normed}).
The vector fields $\partial_{\hat{u}_{2}^{j_{2}}}$, $j_{2}=1,\ldots,m_{2}$
are clearly projectable with $f_{*}(\partial_{\hat{u}_{2}^{j_{2}}})=\partial_{\bar{x}_{2}^{j_{2},+}}$,
and therefore the distribution $D=\mathrm{span}\{\partial_{\hat{u}_{2}}\}$
is an $m_{2}$-dimensional, projectable and involutive subdistribution
of $\mathrm{span}\{\partial_{u}\}$.
\end{proof}
The decomposition of Theorem \ref{thm:decomposition_conditions_distribution}
is a generalization of a decomposition that is used in \cite{Grizzle:1986}
and \cite{NijmeijervanderSchaft:1990} for static feedback linearizable
systems. If a system (\ref{eq:sys}) with $\mathrm{rank}(\partial_{u}f)=m$
is static feedback linearizable, then the complete input distribution
$\mathrm{span}\{\partial_{u}\}$ is projectable. Thus, we can choose
$D=\mathrm{span}\{\partial_{u}\}$. Since this distribution is already
straightened out, no input transformation is required. With a state
transformation that straightens out the pushforward $f_{*}D$, the
system can be transformed into the form
\begin{equation}
\begin{array}{lcl}
\bar{x}_{1}^{i_{1},+}=\bar{f}_{1}^{i_{1}}(\bar{x}_{1},\bar{x}_{2})\,, & \quad & i_{1}=1,\ldots,n-m\\
\bar{x}_{2}^{i_{2},+}=\bar{f}_{2}^{i_{2}}(\bar{x}_{1},\bar{x}_{2},u)\,, &  & i_{2}=1,\ldots,m\,,
\end{array}\label{eq:sys_decomposed_static_feedback_lin}
\end{equation}
where the first $n-m$ equations are independent of all inputs. For
systems that are only flat but not static feedback linearizable, a
transformation into the form (\ref{eq:sys_decomposed_static_feedback_lin})
is in general not possible. However, we will show that a flat system
can always be transformed into the form (\ref{eq:sys_decomposed})
with $m_{2}\geq1$. That is, in the ``worst case'' with $m_{2}=1$
there exists at least a decomposition
\begin{equation}
\begin{array}{ccl}
\bar{x}^{1,+} & = & \bar{f}^{1}(\bar{x},\bar{u}^{1},\ldots,\bar{u}^{m-1})\\
 & \vdots\\
\bar{x}^{n-1,+} & = & \bar{f}^{n-1}(\bar{x},\bar{u}^{1},\ldots,\bar{u}^{m-1})\\
\bar{x}^{n,+} & = & \bar{f}^{n}(\bar{x},\bar{u}^{1},\ldots,\bar{u}^{m-1},\bar{u}^{m})
\end{array}\label{eq:sys_decomposed_1dim}
\end{equation}
where the first $n-1$ equations are independent of $\bar{u}^{m}$.\footnote{For the discussion of the case $m_{2}=1$ we will mainly use the notation
(\ref{eq:sys_decomposed_1dim}) with individual variables instead
of the notation (\ref{eq:sys_decomposed}) with blocks of variables
$\bar{x}_{1},\bar{x}_{2},\bar{u}_{1},\bar{u}_{2}$.} To keep the proof of this remarkable feature of flat systems as short
as possible, it is convenient to rewrite the conditions of Theorem
\ref{thm:decomposition_conditions_distribution} for the case $m_{2}=1$
in terms of $f$-related vector fields instead of distributions.
\begin{cor}
\label{cor:decomposition_conditions_vectorfield}A system (\ref{eq:sys})
with $\mathrm{rank}(\partial_{u}f)=m$ can be transformed into the
form (\ref{eq:sys_decomposed_1dim}) if and only if there exists a
pair of vector fields
\[
v=v^{j}(x,u)\partial_{u^{j}}
\]
on $\mathcal{X}\times\mathcal{U}$ and
\[
w=w^{i}(x^{+})\partial_{x^{i,+}}
\]
on $\mathcal{X}^{+}$ which are $f$-related, i.e., that satisfy
\begin{equation}
w^{i}\circ f(x,u)=\left(\partial_{u^{j}}f^{i}(x,u)\right)v^{j}(x,u)\,,\quad i=1,\ldots,n\,.\label{eq:decomposition_conditions_f-relatedness}
\end{equation}
\end{cor}

\begin{proof}
\ 

\emph{Sufficiency:} Because of the flow-box theorem, there exists
an input transformation
\[
\bar{u}^{j}=\Phi_{u}^{j}(x,u)\,,\quad j=1,\ldots,m
\]
which transforms the vector field $v$ into the form
\[
v=\partial_{\bar{u}^{m}}\,,
\]
and a state transformation
\[
\bar{x}^{i}=\Phi_{x}^{i}(x)\,,\quad i=1,\ldots,n
\]
which transforms the vector field $w$ into the form
\[
w=\partial_{\bar{x}^{n,+}}\,.
\]
In these new coordinates, condition (\ref{eq:decomposition_conditions_f-relatedness})
has the form\footnote{Here $\delta_{n}^{i}$ and $\delta_{m}^{j}$ is the Kronecker delta
and not a shift operator.}
\[
\begin{array}{ccl}
\delta_{n}^{i} & = & \left(\partial_{\bar{u}^{j}}\bar{f}^{i}(\bar{x},\bar{u})\right)\delta_{m}^{j}\\
 & = & \partial_{\bar{u}^{m}}\bar{f}^{i}(\bar{x},\bar{u})
\end{array}\,,\quad i=1,\ldots,n\,.
\]
Because of $\partial_{\bar{u}^{m}}\bar{f}^{i}(\bar{x},\bar{u})=0$
for $i=1,\ldots,n-1$, the functions $\bar{f}^{1},\ldots,\bar{f}^{n-1}$
are independent of $\bar{u}^{m}$.

\emph{Necessity:} If the system is in the form (\ref{eq:sys_decomposed_1dim}),
we can perform an input transformation $\hat{u}^{m}=\bar{f}^{n}(\bar{x},\bar{u})$
such that we get
\[
\begin{array}{ccl}
\bar{x}^{1,+} & = & \bar{f}^{1}(\bar{x},\bar{u}^{1},\ldots,\bar{u}^{m-1})\\
 & \vdots\\
\bar{x}^{n-1,+} & = & \bar{f}^{n-1}(\bar{x},\bar{u}^{1},\ldots,\bar{u}^{m-1})\\
\bar{x}^{n,+} & = & \hat{u}^{m}\,.
\end{array}
\]
In these coordinates, it is obvious that the vector fields $v=\partial_{\hat{u}^{m}}$
and $w=\partial_{\bar{x}^{n,+}}$ are $f$-related.
\end{proof}
The concept of the proofs of Theorem \ref{thm:decomposition_conditions_distribution}
and Corollary \ref{cor:decomposition_conditions_vectorfield} is of
course almost identical. The difference is that in the proof of Corollary
\ref{cor:decomposition_conditions_vectorfield} we straighten out
vector fields with the flow-box theorem, whereas in the proof of Theorem
\ref{thm:decomposition_conditions_distribution} we straighten out
distributions with the Frobenius theorem. The connection between the
distributions of Theorem \ref{thm:decomposition_conditions_distribution}
and the vector fields of Corollary \ref{cor:decomposition_conditions_vectorfield}
is obviously given by
\[
D=\mathrm{span}\{v\}\quad\text{and}\quad f_{*}D=\mathrm{span}\{w\}\,.
\]
In the following, we prove the main result of the paper.
\begin{thm}
\label{thm:decomposition_flat}A flat system (\ref{eq:sys}) with
$\mathrm{rank}(\partial_{u}f)=m$ can be transformed into the form
(\ref{eq:sys_decomposed_1dim}), i.e., (\ref{eq:sys_decomposed})
with $m_{2}=1$.
\end{thm}

\begin{proof}
The proof is based on the identity (\ref{eq:sys_identity_y}). Differentiating
both sides of (\ref{eq:sys_identity_y}) with respect to $y_{[r_{s}]}^{s}$
for some arbitrary $s\in\{1,\ldots,m\}$ gives
\[
\partial_{y_{[r_{s}]}^{s}}\left(\delta_{y}(F_{x}^{i})\right)=\left(\partial_{u^{j}}f^{i}\circ F\right)\partial_{y_{[r_{s}]}^{s}}F_{u}^{j}\,,\quad i=1,\ldots,n\,.
\]
Since $\delta_{y}$ only substitutes variables, shifting and differentiating
with respect to $y_{[r_{s}]}^{s}$ is equivalent to first differentiating
with respect to $y_{[r_{s}-1]}^{s}$ and shifting afterwards. Thus,
we get the equivalent identity
\begin{equation}
\delta_{y}\left(\partial_{y_{[r_{s}-1]}^{s}}F_{x}^{i}\right)=\left(\partial_{u^{j}}f^{i}\circ F\right)\partial_{y_{[r_{s}]}^{s}}F_{u}^{j}\,,\quad i=1,\ldots,n.\label{eq:sys_identity_y_diff_shift}
\end{equation}
Now let us consider this identity in coordinates $(x,u,u_{[1]},\ldots)$.
Substituting (\ref{eq:flat_parametrization_inverse}) into (\ref{eq:sys_identity_y_diff_shift})
gives the identity
\begin{equation}
\delta_{xu}(\tilde{w}^{i}(x,u,u_{[1]},\ldots))=(\partial_{u^{j}}f^{i})\tilde{v}^{j}(x,u,u_{[1]},\ldots)\,.\label{eq:identity_wtilde_vtilde_shift}
\end{equation}
The functions $\tilde{w}^{i}(x,u,u_{[1]},\ldots)$ and $\tilde{v}^{j}(x,u,u_{[1]},\ldots)$
of (\ref{eq:identity_wtilde_vtilde_shift}) are obtained by substituting
(\ref{eq:flat_parametrization_inverse}) into the functions $\partial_{y_{[r_{s}-1]}^{s}}F_{x}^{i}$
and $\partial_{y_{[r_{s}]}^{s}}F_{u}^{j}$ of (\ref{eq:sys_identity_y_diff_shift}).
Note also that substituting (\ref{eq:flat_parametrization_inverse})
into $\partial_{u^{j}}f^{i}\circ F$ yields just $\partial_{u^{j}}f^{i}$,
and that we have to replace the shift operator $\delta_{y}$ in $y$-coordinates
by the shift operator $\delta_{xu}$ in $(x,u)$-coordinates.

Evaluating the expression $\delta_{xu}(w^{i}(x,u,u_{[1]},\ldots))$
on the left-hand side of (\ref{eq:identity_wtilde_vtilde_shift})
yields
\begin{equation}
\tilde{w}^{i}(f(x,u),u_{[1]},u_{[2]},\ldots)=(\partial_{u^{j}}f^{i})\tilde{v}^{j}(x,u,u_{[1]},\ldots)\,.\label{eq:identity_wtilde_vtilde}
\end{equation}
This identity holds (locally) for all values of $x,u,u_{[1]},\ldots$.
Thus, if we evaluate (\ref{eq:identity_wtilde_vtilde})
at any particular point of our underlying space with concrete numerical
values of the coordinates $(x,u,u_{[1]},u_{[2]},\ldots)$, we still
get a valid identity. The same is of course true if we evaluate (\ref{eq:identity_wtilde_vtilde})
on a subspace by setting only some of the coordinates to numerical
values. For our purpose, it is beneficial to evaluate the identity
on a subspace determined by
\begin{equation}
\begin{array}{ccc}
u_{[1]} & = & c_{1}\\
u_{[2]} & = & c_{2}\\
 & \vdots
\end{array}\label{eq:u_forward_constants}
\end{equation}
with arbitrary numerical values for the forward-shifts of $u$ appearing
in (\ref{eq:identity_wtilde_vtilde}).\footnote{Of course the numerical values must be chosen such
that we do not violate localness and still are in a region where the
system is flat. Otherwise, (\ref{eq:identity_wtilde_vtilde}) would
not hold any more. A choice which is sufficiently close to the value
of $u_{0}$ from the equilibrium (\ref{eq:equilibrium}) is always
possible.} By doing so, we get the relation
\[
\tilde{w}^{i}(f(x,u),c_{1},c_{2},\ldots)=(\partial_{u^{j}}f^{i})\tilde{v}^{j}(x,u,c_{1},\ldots)\,.
\]
With
\begin{equation}
w^{i}(x^{+})=\tilde{w}^{i}(x^{+},c_{1},c_{2},\ldots)\label{eq:w_wtilde}
\end{equation}
and
\begin{equation}
v^{j}(x,u)=\tilde{v}^{j}(x,u,c_{1},\ldots)\label{eq:v_vtilde}
\end{equation}
this can be written as
\[
w^{i}(x^{+})\circ f(x,u)=(\partial_{u^{j}}f^{i})v^{j}(x,u)\,,\quad i=1,\ldots,n\,,
\]
which is just condition (\ref{eq:decomposition_conditions_f-relatedness}).
Thus, the vector fields
\begin{equation}
v=v^{j}(x,u)\partial_{u^{j}}\label{eq:proof_f-related_v}
\end{equation}
on $\mathcal{X}\times\mathcal{U}$ and
\begin{equation}
w=w^{i}(x^{+})\partial_{x^{i,+}}\label{eq:proof_f-related_w}
\end{equation}
on $\mathcal{X}^{+}$ are $f$-related. Applying Corollary \ref{cor:decomposition_conditions_vectorfield}
completes the proof.
\end{proof}
\begin{rem}
Let us summarize the idea of the proof once more.
Starting with the identity (\ref{eq:sys_identity_y}), which is a
basic property of every flat system, it is always possible to construct
a pair of $f$-related vector fields (\ref{eq:proof_f-related_v})
and (\ref{eq:proof_f-related_w}). The choice for the numerical values
$c_{1},c_{2},\ldots$, which are used for the construction of these
vector fields, is of course not unique, and for different choices
we get in general different pairs of $f$-related vector fields. However,
as soon as we have any pair of $f$-related vector fields, no matter
how they were constructed, we can straighten them out by the flow-box
theorem and get a state transformation and an input transformation
which transforms the system into the decomposed form (\ref{eq:sys_decomposed_1dim}),
cf. Corollary \ref{cor:decomposition_conditions_vectorfield}. The
transformed system equations (\ref{eq:sys_decomposed_1dim}) are just
as general as the original ones, and of course not restricted to input
sequences with the numerical values (\ref{eq:u_forward_constants})
used for the construction of the vector fields (\ref{eq:proof_f-related_v})
and (\ref{eq:proof_f-related_w}). Substituting numerical values (\ref{eq:u_forward_constants})
into the identity (\ref{eq:identity_wtilde_vtilde}) is just a useful
operation to construct the vector fields (\ref{eq:proof_f-related_v})
and (\ref{eq:proof_f-related_w}), but does not restrict the validity
of the transformed system equations (\ref{eq:sys_decomposed_1dim})
obtained by straightening these vector fields out.
\end{rem}

As a consequence of Theorem \ref{thm:decomposition_flat}, the existence
of a decomposed form (\ref{eq:sys_decomposed_1dim}) is a necessary
condition for flat discrete-time systems. Based on similar ideas as
in the proof of Theorem \ref{thm:decomposition_flat}, it has been
shown in \cite{KolarSchoberlSchlacher:2016-3} that for flat continuous-time
systems
\[
\dot{x}^{i}=f^{i}(x,u)\,,\quad i=1,\ldots,n
\]
there always exists a transformation $\bar{u}=\Phi_{u}(x,u)$ into
the so-called partial affine input form (PAI-form)
\begin{equation}
\dot{x}^{i}=a^{i}(x,\bar{u}^{1},\ldots,\bar{u}^{m-1})+b^{i}(x,\bar{u}^{1},\ldots,\bar{u}^{m-1})\bar{u}^{m}\,,\label{eq:PAI-form}
\end{equation}
$i=1,\ldots,n$, where $\bar{u}^{m}$ appears in an affine way. This
PAI-form is closely related to the well-known ruled manifold necessary
condition derived in \cite{Rouchon:1994} for flat continuous-time
systems. Thus, the existence of the decomposed form (\ref{eq:sys_decomposed_1dim})
for flat discrete-time systems can be interpreted as discrete-time
counterpart to the existence of a PAI-form (\ref{eq:PAI-form}) for
flat continuous-time systems.

\section{\label{sec:Algorithm}Calculation of Flat Outputs}

We show in this section that a repeated application of the results
of Section \ref{sec:Decomposition} gives rise to an algorithm, which
allows to check the flatness of a discrete-time system (\ref{eq:sys})
with $\mathrm{rank}(\partial_{u}f)=m$ in at most $n-1$ steps. If
the system is flat, then the algorithm provides a flat output. Otherwise,
it stops and we can conclude that the system is not flat. Roughly
speaking, the idea is as follows: If the system (\ref{eq:sys}) is
flat, then Theorem \ref{thm:decomposition_flat} guarantees that it
can be transformed into the form (\ref{eq:sys_decomposed}) with an
at most $(n-1)$-dimensional subsystem $\bar{x}_{1}^{+}=\bar{f}_{1}(\bar{x}_{1},\bar{x}_{2},\bar{u}_{1})$.
Because of Lemma \ref{lem:basic_decomposition_flat} this subsystem
is also flat, and therefore Theorem \ref{thm:decomposition_flat}
guarantees that the subsystem can again be transformed into the form
(\ref{eq:sys_decomposed}). Repeating this procedure reduces the problem
of checking the flatness of the original system (\ref{eq:sys}) to
the problem of checking the flatness of smaller and smaller subsystems.
Obviously, for a system (\ref{eq:sys}) with $\dim(x)=n$ we can perform
at most $n-1$ such decomposition steps. If in some
step we encounter a subsystem with the same number of input and state
variables, then we can read off a flat output of this subsystem (the
state variables), and the original system is also flat. Otherwise,
if we find a subsystem which does not allow a further decomposition,
then Theorem \ref{thm:decomposition_flat} implies that this subsystem
is not flat. Therefore, the original system (\ref{eq:sys}) cannot
be flat either.

What we have not mentioned in this brief sketch of the basic idea
is the fact that there may appear subsystems with redundant inputs,
i.e., where the Jacobian matrix with respect to the inputs of the
subsystem does not have full rank (see Lemma \ref{lem:basic_decomposition_flat}
and Remark \ref{rem:redundant_inputs_state_variables}). In this
case, we have to eliminate these redundant inputs with a suitable
coordinate transformation, before we can apply Theorem \ref{thm:decomposition_conditions_distribution}
to construct a decomposition of the subsystem.
\begin{rem}
This effect is well-known from static feedback linearization, see
e.g. \cite{NijmeijervanderSchaft:1990}. For instance, if a static
feedback linearizable system is transformed into the form (\ref{eq:sys_decomposed_static_feedback_lin}),
it may happen that $\mathrm{rank}(\partial_{\bar{x}_{2}}\bar{f}_{1})<m$.
\end{rem}

The elimination of redundant inputs is, however, very easy: For a
system (\ref{eq:sys}) with $\mathrm{rank}(\partial_{u}f)=\hat{m}<m$
there always exists an input transformation $(\hat{u},\tilde{u})=\Phi_{u}(x,u)$
with $\dim(\hat{u})=\hat{m}$ that eliminates $m-\hat{m}$ redundant
inputs $\tilde{u}$. If $u=\hat{\Phi}_{u}(x,\hat{u},\tilde{u})$ denotes
the inverse input transformation, then the transformed system is of
the form
\begin{equation}
x^{i,+}=\hat{f}^{i}(x,\hat{u})\,,\quad i=1,\ldots,n\label{eq:redundant_inputs_transformed_sys}
\end{equation}
with
\[
\hat{f}^{i}(x,\hat{u})=f^{i}(x,\hat{\Phi}_{u}(x,\hat{u},\tilde{u}))\,,\quad i=1,\ldots,n
\]
and $\mathrm{rank}(\partial_{\hat{u}}\hat{f})=\hat{m}$. The following
lemma establishes an important connection between a flat output of
the transformed system (\ref{eq:redundant_inputs_transformed_sys})
with $\hat{m}$ inputs and the original system (\ref{eq:sys}) with
$m$ inputs.
\begin{lem}
\label{lem:flatness_redundant_inputs}Consider a system (\ref{eq:sys})
with $\mathrm{rank}(\partial_{u}f)=\hat{m}<m$, and an input transformation
$(\hat{u},\tilde{u})=\Phi_{u}(x,u)$ with $\dim(\hat{u})=\hat{m}$
that eliminates $m-\hat{m}$ redundant inputs $\tilde{u}$. If an
$\hat{m}$-tuple $\hat{y}$ is a flat output of the transformed system
(\ref{eq:redundant_inputs_transformed_sys}) with the $\hat{m}$ inputs
$\hat{u}$, then the $m$-tuple $y=(\hat{y},\tilde{u})$ is a flat
output of the original system (\ref{eq:sys}) with the $m$ inputs
$u$.
\end{lem}

\begin{proof}
Since $\hat{y}$ is a flat output of the transformed system (\ref{eq:redundant_inputs_transformed_sys}),
$x$ and $\hat{u}$ can be expressed as functions of $\hat{y}$ and
its forward-shifts. Because of $y=(\hat{y},\tilde{u})$, the inverse
input transformation $u=\hat{\Phi}_{u}(x,\hat{u},\tilde{u})$ shows
immediately that the input $u$ of the original system (\ref{eq:sys})
can be expressed by $y$ and its forward-shifts.
\end{proof}
Thus, eliminated redundant inputs are candidates for components of
a flat output.

Now we can describe the algorithm in detail. To enhance
the readability, every step is divided into three subtasks: (A) checking
whether the (sub-)system is flat by a simple dimension argument, (B)
checking whether a decomposition is possible, (C) performing the decomposition.
For the decompositions, we use the more general formulation of Theorem
\ref{thm:decomposition_conditions_distribution} with distributions,
instead of the 1-dimensional special case of Corollary \ref{cor:decomposition_conditions_vectorfield}
with vector fields. To keep the notation as simple
as possible, after every decomposition step the state and input of
the remaining subsystem are renamed again as $x$ and $u$.\addtocounter{thm}{1}

\textbf{Algorithm \thethm}\emph{
Start with the original system (\ref{eq:sys}) and the first decomposition
step with $k=1$. }\emph{We assume that the original
system meets $\mathrm{rank}(\partial_{u}f)=m$, i.e., has no redundant
inputs.}

\textbf{\emph{Decomposition Step $k\geq1$:}}
\begin{enumerate}
\item[(A)] \emph{If $\dim(x)>\dim(u)$, go
to (B). Otherwise, $y_{k}=x$ is a flat output of the system considered
in the $k$-th decomposition step. A flat output of the original system
is given by $y=(y_{k},\ldots,y_{1})$. This follows immediately from
a $(k-1)$-fold application of Lemma~\ref{lem:basic_decomposition_flat}
and Lemma \ref{lem:flatness_redundant_inputs}.}
\item[(B)] \emph{Transform the input vector
fields $\partial_{u}$ into adapted coordinates (\ref{eq:adapted_coordinates}),
and check whether there exists a projectable and involutive subdistribution
$D\subset\mathrm{span}\{\partial_{u}\}$. In case of a positive result,
go to (C). In case of a negative result, according to Theorem \ref{thm:decomposition_flat}
the system considered in the $k$-th decomposition step is not flat.}\footnote{In fact, the formulation of Theorem \ref{thm:decomposition_flat}
guarantees already the existence of a decomposition for flat systems.
However, according to Theorem \ref{thm:decomposition_conditions_distribution},
this is equivalent to the existence of an at least 1-dimensional projectable
and involutive subdistribution $D\subset\mathrm{span}\{\partial_{u}\}$.}\emph{ By a $(k-1)$-fold application of Lemma \ref{lem:basic_decomposition_flat},
the original system cannot be flat either.}
\item[(C)] \emph{Straighten out $D$ by an
input transformation (\ref{eq:decompostion_input_transformation}),
and the pushforward $f_{*}D$ by a state transformation (\ref{eq:decompostion_state_transformation}).
In new coordinates, the system is of the form (\ref{eq:sys_decomposed}).
Now consider the subsystem}
\begin{equation}
\bar{x}_{1}^{+}=\bar{f}_{1}(\bar{x}_{1},\bar{x}_{2},\bar{u}_{1})\label{eq:algorithm_subsys_bar}
\end{equation}
\emph{with the inputs $(\bar{x}_{2},\bar{u}_{1})$,
and eliminate redundant inputs by a coordinate transformation
\[
(\hat{z},y_{k})=\Phi_{u}(\bar{x}_{1},\bar{x}_{2},\bar{u}_{1})
\]
with $\dim(\hat{z})=\mathrm{rank}(\partial_{(\bar{x}_{2},\bar{u}_{1})}\bar{f}_{1})$,
such that the transformed system}
\begin{equation}
\bar{x}_{1}^{+}=\hat{f}_{1}(\bar{x}_{1},\hat{z})\label{eq:algorithm_subsys_bar_hat}
\end{equation}
\emph{does not depend on $y_{k}$.}\footnote{Note that, as discussed in Lemma \ref{lem:flatness_redundant_inputs},
adding $y_{k}$ to a flat output of (\ref{eq:algorithm_subsys_bar_hat})
yields a flat output of (\ref{eq:algorithm_subsys_bar}).}\emph{ According to Remark \ref{rem:redundant_inputs_state_variables},
because of $\mathrm{rank}(\partial_{\bar{u}_{1}}\bar{f}_{1})=\dim(\bar{u}_{1})$
the components of $y_{k}$ can always be found among the state variables
$\bar{x}_{2}$. In the case $\mathrm{rank}(\partial_{(\bar{x}_{2},\bar{u}_{1})}\bar{f}_{1})=\dim(\bar{x}_{2})+\dim(\bar{u}_{1})$,
$y_{k}$ is empty. Finally, rename the subsystem (\ref{eq:algorithm_subsys_bar_hat})
as $x^{+}=f(x,u)$ with
\[
x=\bar{x}_{1}\,,\quad u=\hat{z}\,,\quad f=\hat{f}_{1}\,,
\]
and proceed with item (A) of the next decomposition step.}
\end{enumerate}
In the following, we want to discuss computational
aspects of the algorithm. First, checking the existence of a projectable
and involutive subdistribution $D\subset\mathrm{span}\{\partial_{u}\}$
in item (B) requires only the solution of algebraic equations: To
transform the input vector fields $\partial_{u}$ into adapted coordinates
(\ref{eq:adapted_coordinates}), according to the transformation law
for vector fields we also need the inverse of (\ref{eq:adapted_coordinates}).
In adapted coordinates, the input vector fields $\partial_{u}$ are
of the form (\ref{eq:f-related_v_adapt}). Then we only have to check
if there exists at least one linear combination (with coefficients
that may depend on $x^{+}$ and $\xi$) which is of the form
\begin{equation}
a^{i}(x^{+})\partial_{x^{i,+}}+b^{j}(x^{+},\xi)\partial_{\xi^{j}}\,,\label{eq:projectable_vector_field}
\end{equation}
i.e., projectable. For further computational details
see the appendix. Every such vector field spans
a 1-dimensional (and thus involutive) projectable subdistribution
$D\subset\mathrm{span}\{\partial_{u}\}$. The vector field (\ref{eq:projectable_vector_field})
in coordinates $(x,u)$ follows from the inverse of (\ref{eq:adapted_coordinates}).
\begin{rem}
\label{rem:decomposition_not_unique}Note that in
general the choice of an involutive distribution $D$ that meets the
conditions of Theorem \ref{thm:decomposition_conditions_distribution},
or equivalently a pair of vector fields that meets the conditions
of Corollary \ref{cor:decomposition_conditions_vectorfield}, is not
unique. Thus, the decomposition (\ref{eq:sys_decomposed}) is not
uniquely determined.
\end{rem}

Second, in item (C) the distribution $D$ and its
pushforward $f_{*}D$ have to be straightened out by an input transformation
and a state transformation. Since straightening out involutive distributions
by the Frobenius theorem requires the solution of (nonlinear) ODEs,
this task is typically considerably more difficult than the construction
of the distributions in item (B). However, for the calculation of
a linearizing output of a static feedback linearizable system it is
also necessary to straighten out a sequence of distributions by the
Frobenius theorem. Thus, from a computational point of view, the construction
of a flat output is essentially of the same complexity as the construction
of a linearizing output of a static feedback linearizable system.
The main difference is that we have to solve additionally algebraic
equations to determine a suitable subdistribution $D\subset\mathrm{span}\{\partial_{u}\}$,
whereas in the static feedback linearization problem we always work
with the complete distribution $D=\mathrm{span}\{\partial_{u}\}$.

The algorithm is in fact a generalization of the transformation of
static feedback linearizable systems into a triangular form which
is discussed in \cite{NijmeijervanderSchaft:1990}. The transformation
into this triangular form can be interpreted as a repeated application
of the decomposition (\ref{eq:sys_decomposed_static_feedback_lin}),
and yields a linearizing output.\footnote{Note that in \cite{NijmeijervanderSchaft:1990} all decomposition
steps are combined in one coordinate transformation, which is obtained
by straightening out a nested sequence of involutive distributions.} For the calculation of flat outputs, we simply have to replace the
decomposition (\ref{eq:sys_decomposed_static_feedback_lin}) by the
more general decomposition (\ref{eq:sys_decomposed}). However, it
is important to emphasize that the decompositions we perform in each
of the steps are typically not unique, and that different decompositions
might lead to different flat outputs. This is in accordance with the
fact that flat outputs (of multi-input systems) are never unique.
It is also obvious that every flat output which is obtained by the
suggested algorithm can only depend on $x$ and $u$ but not on forward-shifts
of $u$. This is a simple consequence of the fact
that we do not introduce any additional variables. Since the algorithm
yields (in principle) a flat output for every flat discrete-time system,
we can conclude that every flat discrete-time system must have a flat
output which only depends on $x$ and $u$. By a closer inspection,
we get an even stronger result.
\begin{thm}
\label{thm:x-flat}Every flat discrete-time system (\ref{eq:sys})
with $\mathrm{rank}(\partial_{u}f)=m$ has a flat output of the form
$y=\varphi(x)$, which is independent of the input $u$ and its forward-shifts.
\end{thm}

\begin{proof}
Suppose the algorithm terminates after $k=\bar{k}$
steps. Then the constructed flat output is of the form
\[
y=(y_{\bar{k}},\ldots,y_{1})\,,
\]
where $y_{\bar{k}}$ consists of the state variables of the last subsystem,
and $y_{\bar{k}-1},\ldots,y_{1}$ are eliminated redundant inputs
of the subsystems constructed in the first $\bar{k}-1$ decomposition
steps. Thus, input variables of the original system could only appear
in the components $y_{\bar{k}-1},\ldots,y_{1}$. However, as discussed
in Remark \ref{rem:redundant_inputs_state_variables} and item (C)
of the algorithm, due to the full rank of the Jacobian matrix $\partial_{u}f$
the redundant inputs of the subsystems (\ref{eq:algorithm_subsys_bar})
can always be found among the (transformed) state variables of the
complete system. Thus, the flat output depends indeed only on the
state variables.
\end{proof}
This result is also remarkable, since Theorem \ref{thm:x-flat} does
not have a counterpart for flat continuous-time systems.

\section{\label{sec:Examples}Examples}

In this section, we illustrate our results with two
examples.

\subsection{An Academic Example}

In the following, we demonstrate the algorithm for the calculation
of flat outputs with the system
\begin{equation}
\begin{array}{lcl}
x^{1,+} & = & \tfrac{x^{2}+x^{3}+3x^{4}}{u^{1}+2u^{2}+1}\\
x^{2,+} & = & x^{1}(x^{3}+1)(u^{1}+2u^{2}-3)+x^{4}-3u^{2}\\
x^{3,+} & = & u^{1}+2u^{2}\\
x^{4,+} & = & x^{1}(x^{3}+1)+u^{2}\,.
\end{array}\label{eq:Example1_sys}
\end{equation}
All coordinate transformations that we will perform
are defined in a neighborhood of the equilibrium $(x_{0},u_{0})=(0,0)$.

In the first step, we have to check the existence of a projectable
involutive subdistribution $D\subset\mathrm{span}\{\partial_{u}\}$.
For this purpose, we introduce adapted coordinates (\ref{eq:adapted_coordinates})
on $\mathcal{X}\times\mathcal{U}$. After the transformation
\[
\begin{array}{ccl}
x^{1,+} & = & f^{1}(x,u)\\
x^{2,+} & = & f^{2}(x,u)\\
x^{3,+} & = & f^{3}(x,u)\\
x^{4,+} & = & f^{4}(x,u)\\
\xi^{1} & = & x^{1}\\
\xi^{2} & = & x^{3}\,,
\end{array}
\]
the vector fields $\partial_{u^{1}}$ and $\partial_{u^{2}}$ are
given by
\[
-\tfrac{x^{1,+}}{x^{3,+}+1}\partial_{x^{1,+}}+\xi^{1}(\xi^{2}+1)\partial_{x^{2,+}}+\partial_{x^{3,+}}
\]
and
\[
-2\tfrac{x^{1,+}}{x^{3,+}+1}\partial_{x^{1,+}}+(2\xi^{1}(\xi^{2}+1)-3)\partial_{x^{2,+}}+2\partial_{x^{3,+}}+\partial_{x^{4,+}}\,.
\]
Because of the presence of the fibre coordinates $\xi^{1}$ and $\xi^{2}$,
neither $\partial_{u^{1}}$ nor $\partial_{u^{2}}$ itself is projectable.
However, the linear combination $-2\partial_{u^{1}}+\partial_{u^{2}}$
reads in adapted coordinates as
\[
-3\partial_{x^{2,+}}+\partial_{x^{4,+}}
\]
and is hence a projectable vector field. Since there is no other
possibility (besides a scaling), the distribution
$D=\mathrm{span}\{-2\partial_{u^{1}}+\partial_{u^{2}}\}$ is uniquely
determined, and the pushforward yields $f_{*}D=\mathrm{span}\{-3\partial_{x^{2,+}}+\partial_{x^{4,+}}\}$.
Because of $\dim(D)<2$, the system (\ref{eq:Example1_sys}) cannot
be static feedback linearizable. Now we straighten out the involutive
distributions $D$ and $f_{*}D$ by input- and state transformations.
The input transformation
\[
\begin{array}{ccl}
\bar{u}^{1} & = & u^{1}+2u^{2}\\
\bar{u}^{2} & = & u^{2}
\end{array}
\]
gives $D=\mathrm{span}\{\partial_{\bar{u}^{2}}\}$, and the state
transformation
\[
\begin{array}{ccl}
\bar{x}^{1} & = & x^{1}\\
\bar{x}^{2} & = & x^{2}+3x^{4}\\
\bar{x}^{3} & = & x^{3}\\
\bar{x}^{4} & = & x^{4}
\end{array}
\]
gives $f_{*}D=\mathrm{span}\{\partial_{\bar{x}^{4,+}}\}$. Accordingly,
the transformed system reads
\[
\begin{array}{lcl}
\bar{x}^{1,+} & = & \tfrac{\bar{x}^{2}+\bar{x}^{3}}{\bar{u}^{1}+1}\\
\bar{x}^{2,+} & = & \bar{x}^{1}(\bar{x}^{3}+1)\bar{u}^{1}+\bar{x}^{4}\\
\bar{x}^{3,+} & = & \bar{u}^{1}\\
\bar{x}^{4,+} & = & \bar{x}^{1}(\bar{x}^{3}+1)+\bar{u}^{2}\,,
\end{array}
\]
where the first three equations are independent of $\bar{u}^{2}$.
Now consider the subsystem
\begin{equation}
\begin{array}{lcl}
\bar{x}^{1,+} & = & \tfrac{\bar{x}^{2}+\bar{x}^{3}}{\bar{u}^{1}+1}\\
\bar{x}^{2,+} & = & \bar{x}^{1}(\bar{x}^{3}+1)\bar{u}^{1}+\bar{x}^{4}\\
\bar{x}^{3,+} & = & \bar{u}^{1}
\end{array}\label{eq:Example1_subsys1}
\end{equation}
with the inputs $(\bar{x}^{4},\bar{u}^{1})$. Because of $\mathrm{rank}(\partial_{(\bar{x}^{4},\bar{u}^{1})}\bar{f})=2$,
where by $\bar{f}$ we refer to the system (\ref{eq:Example1_subsys1}),
there are no redundant inputs. Thus, $y_{1}=\{\}$
is empty.

In the second step, the complete input distribution 
of the system (\ref{eq:Example1_subsys1}) is projectable
(this can be verified again by introducing adapted coordinates) and
we can choose $D=\mathrm{span}\{\partial_{\bar{x}^{4}},\partial_{\bar{u}^{1}}\}$,
which is clearly involutive. Since this distribution is already straightened
out, we need no input transformation, i.e., we can simply set
\[
\begin{array}{ccl}
\bar{\bar{x}}^{4} & = & \bar{x}^{4}\\
\bar{\bar{u}}^{1} & = & \bar{u}^{1}\,.
\end{array}
\]
The pushforward of $D$ is given by
\[
\bar{f}_{*}D=\mathrm{span}\{\partial_{\bar{x}^{2,+}},-\tfrac{\bar{x}^{1,+}}{\bar{x}^{3,+}+1}\partial_{\bar{x}^{1,+}}+\partial_{\bar{x}^{3,+}}\}\,,
\]
and the state transformation
\[
\begin{array}{ccl}
\bar{\bar{x}}^{1} & = & \bar{x}^{1}(\bar{x}^{3}+1)\\
\bar{\bar{x}}^{2} & = & \bar{x}^{2}\\
\bar{\bar{x}}^{3} & = & \bar{x}^{3}
\end{array}
\]
yields $\bar{f}_{*}D=\mathrm{span}\{\partial_{\bar{\bar{x}}^{2,+}},\partial_{\bar{\bar{x}}^{3,+}}\}$.
In new coordinates, the system reads
\[
\begin{array}{lcl}
\bar{\bar{x}}^{1,+} & = & \bar{\bar{x}}^{2}+\bar{\bar{x}}^{3}\\
\bar{\bar{x}}^{2,+} & = & \bar{\bar{x}}^{1}\bar{\bar{u}}^{1}+\bar{\bar{x}}^{4}\\
\bar{\bar{x}}^{3,+} & = & \bar{\bar{u}}^{1}\,,
\end{array}
\]
where the first line is independent of both inputs $\bar{\bar{x}}^{4}$
and $\bar{\bar{u}}^{1}$. The subsystem
\begin{equation}
\begin{array}{lcl}
\bar{\bar{x}}^{1,+} & = & \bar{\bar{x}}^{2}+\bar{\bar{x}}^{3}\end{array}\label{eq:Example1_subsys2}
\end{equation}
with the inputs $(\bar{\bar{x}}^{2},\bar{\bar{x}}^{3})$
meets $\mathrm{rank}(\partial_{(\bar{\bar{x}}^{2},\bar{\bar{x}}^{3})}\bar{\bar{f}})=1<2$.
Thus, there exists a redundant input. The elimination
of a redundant input is obviously not unique. Possible choices are
e.g. the transformations
\[
\begin{array}{ccl}
\hat{z} & = & \bar{\bar{x}}^{2}+\bar{\bar{x}}^{3}\\
y_{2} & = & \bar{\bar{x}}^{2}
\end{array}
\]
or
\[
\begin{array}{ccl}
\hat{z} & = & \bar{\bar{x}}^{2}+\bar{\bar{x}}^{3}\\
y_{2} & = & \bar{\bar{x}}^{3}\,.
\end{array}
\]
In both cases, the transformed system (\ref{eq:Example1_subsys2})
reads
\begin{equation}
\begin{array}{lcl}
\bar{\bar{x}}^{1,+} & = & \hat{z}\end{array}\,.\label{eq:Example1_subsys2_elim}
\end{equation}

In the third step, we finally have a system with
the same number of input- and state variables. Thus, a flat output
of (\ref{eq:Example1_subsys2_elim}) is given by $y_{3}=\bar{\bar{x}}^{1}$.
Adding the redundant input $y_{2}$ yields a flat output of (\ref{eq:Example1_subsys2}),
which is also a flat output of the complete system (\ref{eq:Example1_sys}).
In original coordinates, the flat output $y=(\bar{\bar{x}}^{1},\bar{\bar{x}}^{2})$
is given by $y=(x^{1}(x^{3}+1),\,x^{2}+3x^{4})$, and the flat output
$y=(\bar{\bar{x}}^{1},\bar{\bar{x}}^{3})$ is given by $y=(x^{1}(x^{3}+1),\,x^{3})$.

For the flat output
\[
y=(x^{1}(x^{3}+1),\,x^{2}+3x^{4})\,,
\]
the map (\ref{eq:flat_parametrization}) is given by
\[
\begin{array}{ccl}
x^{1} & = & \tfrac{y^{1}}{y_{[1]}^{1}-y^{2}+1}\\
x^{2} & = & 3y^{1}(y_{[2]}^{1}-y_{[1]}^{2})+y^{2}-3y_{[1]}^{2}\\
x^{3} & = & y_{[1]}^{1}-y^{2}\\
x^{4} & = & y^{1}(y_{[1]}^{2}-y_{[2]}^{1})+y_{[1]}^{2}\\
u^{1} & = & 2y^{1}+2y_{[1]}^{1}(y_{[3]}^{1}-y_{[2]}^{2})+y_{[2]}^{1}-y_{[1]}^{2}-2y_{[2]}^{2}\\
u^{2} & = & -y^{1}+y_{[1]}^{1}(y_{[2]}^{2}-y_{[3]}^{1})+y_{[2]}^{2}\,.
\end{array}
\]
That is, there appear forward-shifts of $y^{1}$ and $y^{2}$ up to
the orders $r_{1}=3$ and $r_{2}=2$. In the following, we shall use
this example to illustrate the method that we have applied in the
proof of Theorem \ref{thm:decomposition_flat} to show that every
flat system allows a decomposition (\ref{eq:sys_decomposed_1dim}).
Since the $1$-dimensional distributions $D$ and $f_{*}D$ in the
first decomposition step of the system (\ref{eq:Example1_sys}) are
unique, the method of Theorem \ref{thm:decomposition_flat} must yield
exactly the same decomposition. Substituting (\ref{eq:flat_parametrization_inverse})
into
\[
\partial_{y_{[2]}^{1}}F_{x}^{i}\,,\quad i=1,\ldots,4
\]
and
\[
\partial_{y_{[3]}^{1}}F_{u}^{j}\,,\quad j=1,2
\]
yields
\[
\begin{array}{ccl}
\tilde{w}^{1} & = & 0\\
\tilde{w}^{2} & = & 3x^{1}(x^{3}+1)\\
\tilde{w}^{3} & = & 0\\
\tilde{w}^{4} & = & -x^{1}(x^{3}+1)
\end{array}
\]
and
\[
\begin{array}{ccl}
\tilde{v}^{1} & = & 2(x^{2}+x^{3}+3x^{4})\\
\tilde{v}^{2} & = & -(x^{2}+x^{3}+3x^{4})\,.
\end{array}
\]
Since the functions $\tilde{w}$ are independent of $u$ and the functions
$\tilde{v}$ are independent of forward-shifts of $u$, we directly
get $w^{i}=\tilde{w}^{i}$ and $v^{j}=\tilde{v}^{j}$. It can be checked
easily that the condition (\ref{eq:decomposition_conditions_f-relatedness})
is indeed satisfied. Therefore, the pair of vector fields
\[
v=2(x^{2}+x^{3}+3x^{4})\partial_{u^{1}}-(x^{2}+x^{3}+3x^{4})\partial_{u^{2}}
\]
and
\[
w=3x^{1,+}(x^{3,+}+1)\partial_{x^{2,+}}-x^{1,+}(x^{3,+}+1)\partial_{x^{4,+}}
\]
is $f$-related. Because of
\begin{equation}
\mathrm{span}\{v\}=\mathrm{span}\{-2\partial_{u^{1}}+\partial_{u^{2}}\}=D\label{eq:Example1_F_D}
\end{equation}
and
\begin{equation}
\mathrm{span}\{w\}=\mathrm{span}\{-3\partial_{x^{2,+}}+\partial_{x^{4,+}}\}=f_{*}D\,,\label{eq:Example1_F_fstarD}
\end{equation}
these vector fields span exactly the same distributions that we have
constructed in the first decomposition step in adapted coordinates.
\begin{rem}
It should be noted that the case $w^{i}=\tilde{w}^{i}$ and $v^{j}=\tilde{v}^{j}$
is a special one and does not hold in general. With the more sophisticated
flat output $y=(x^{1}(x^{3}+1)+e^{u^{1}+2u^{2}},x^{3})$, we would
get functions $\tilde{w}^{i}$ and $\tilde{v}^{j}$ that also depend
on forward-shifts of $u$. After setting these forward-shifts to constant
values as shown in (\ref{eq:w_wtilde}) and (\ref{eq:v_vtilde}),
the resulting $f$-related vector fields
\[
v=e^{c_{3}^{1}+2c_{3}^{2}}(2\partial_{u^{1}}-\partial_{u^{2}})
\]
and
\[
w=e^{c_{3}^{1}+2c_{3}^{2}}(3\partial_{x^{2,+}}-\partial_{x^{4,+}})
\]
span again the same distributions (\ref{eq:Example1_F_D}) and (\ref{eq:Example1_F_fstarD}),
independent of the chosen values $c_{3}^{1}$ and $c_{3}^{2}$. This
is a consequence of the fact that system (\ref{eq:Example1_sys})
possesses only a 1-dimensional projectable subdistribution $D\subset\mathrm{span}\{\partial_{u}\}$.
\end{rem}

\subsection{A Wheeled Mobile Robot}

As a second example, we consider the exact discretization
of the kinematic model
\begin{equation}
\begin{array}{ccl}
\dot{x}^{1} & = & \sin(x^{3})u^{1}\\
\dot{x}^{2} & = & \cos(x^{3})u^{1}\\
\dot{x}^{3} & = & u^{2}
\end{array}\label{eq:WMR}
\end{equation}
of a wheeled mobile robot, which is also discussed in the context
of dynamic feedback linearization in \cite{Aranda-BricaireMoog:2008}.
The variables $x^{1}$ and $x^{2}$ describe the position of the center
of the axle, and $x^{3}$ its orientation. The control inputs are
the translatory velocity $u^{1}$ and the angular velocity $u^{2}$.
It is well-known that the continuous-time system (\ref{eq:WMR}) is
flat, and a flat output is given by $y=(x^{1},x^{2})$, i.e., by the
position of the axle.

With the assumption that the inputs $u^{1}$ and
$u^{2}$ are constant between sampling instants, the system (\ref{eq:WMR})
can be solved analytically, and an exact discrete-time model is given
by
\begin{equation}
\begin{array}{ccl}
x^{1,+} & = & x^{1}+2u^{1}\psi(u^{2})\cos(\gamma(x^{3},u^{2}))\\
x^{2,+} & = & x^{2}+2u^{1}\psi(u^{2})\sin(\gamma(x^{3},u^{2}))\\
x^{3,+} & = & x^{3}+Tu^{2}
\end{array}\label{eq:WMR_exact}
\end{equation}
with
\[
\psi(u^{2})=\begin{cases}
\tfrac{\sin\left(\tfrac{T}{2}u^{2}\right)}{u^{2}}\,, & \text{if}\:u^{2}\neq0\\
\tfrac{T}{2}\,, & \text{if}\:u^{2}=0
\end{cases}
\]
and
\[
\gamma(x^{3},u^{2})=x^{3}+\tfrac{T}{2}u^{2}\,,
\]
see \cite{Aranda-BricaireMoog:2008}. As shown in \cite{Aranda-BricaireMoog:2008},
the system (\ref{eq:WMR_exact}) can be linearized by an exogenous
dynamic feedback. In the following, we prove that the system is not
flat, and can thus indeed not be linearized by an endogenous dynamic
feedback. Before we apply our algorithm, we perform the input transformation
\[
\begin{array}{ccl}
\bar{u}^{1} & = & 2u^{1}\psi(u^{2})\\
\bar{u}^{2} & = & \gamma(x^{3},u^{2})
\end{array}
\]
to obtain the simpler system representation
\[
\begin{array}{ccl}
x^{1,+} & = & x^{1}+\bar{u}^{1}\cos(\bar{u}^{2})\\
x^{2,+} & = & x^{2}+\bar{u}^{1}\sin(\bar{u}^{2})\\
x^{3,+} & = & -x^{3}+2\bar{u}^{2}\,.
\end{array}
\]
Now let us check the existence of a projectable involutive subdistribution
$D\subset\mathrm{span}\{\partial_{\bar{u}}\}$. After introducing
adapted coordinates
\[
\begin{array}{ccl}
x^{1,+} & = & x^{1}+\bar{u}^{1}\cos(\bar{u}^{2})\\
x^{2,+} & = & x^{2}+\bar{u}^{1}\sin(\bar{u}^{2})\\
x^{3,+} & = & -x^{3}+2\bar{u}^{2}\\
\xi^{1} & = & x^{3}\\
\xi^{2} & = & \bar{u}^{1}
\end{array}
\]
on $\mathcal{X}\times\mathcal{U}$, the vector fields $\partial_{\bar{u}^{1}}$
and $\partial_{\bar{u}^{2}}$ are given by
\[
\cos\left(\tfrac{x^{3,+}+\xi^{1}}{2}\right)\partial_{x^{1,+}}+\sin\left(\tfrac{x^{3,+}+\xi^{1}}{2}\right)\partial_{x^{2,+}}+\partial_{\xi^{2}}
\]
and
\[
-\xi^{2}\sin\left(\tfrac{x^{3,+}+\xi^{1}}{2}\right)\partial_{x^{1,+}}+\xi^{2}\cos\left(\tfrac{x^{3,+}+\xi^{1}}{2}\right)\partial_{x^{2,+}}+2\partial_{x^{3,+}}\,.
\]
With a normalized basis of the form (\ref{eq:basis_normed}),
it can be observed that there does not exist any projectable linear
combination of these vector fields (there does not exist any linear
combination where the coefficients of $\partial_{x^{1,+}}$, $\partial_{x^{2,+}}$
and $\partial_{x^{3,+}}$ are independent of $\xi^{1}$ and $\xi^{2}$).
Thus, the algorithm stops already in the first step with a negative
result. Since every flat discrete-time system possesses an at least
1-dimensional projectable subdistribution $D\subset\mathrm{span}\{\partial_{u}\}$,
the exact discretization (\ref{eq:WMR_exact}) of the wheeled mobile
robot (\ref{eq:WMR}) is not flat.

An Euler discretization, in contrast, would preserve
the flatness and even the flat output $y=(x^{1},x^{2})$ of the continuous-time
system (\ref{eq:WMR}). In fact, the Euler discretization
\[
\begin{array}{ccl}
x^{1,+} & = & x^{1}+T\sin(x^{3})u^{1}\\
x^{2,+} & = & x^{2}+T\cos(x^{3})u^{1}\\
x^{3,+} & = & x^{3}+Tu^{2}
\end{array}
\]
is already in the decomposed form (\ref{eq:basic_decomposition_flat})
with $m_{2}=1$. Thus, the flat output $y=(x^{1},x^{2})$ can be read
off directly from the system equations. However, it is important to
emphasize that in general also an Euler discretization does not necessarily
preserve the flatness of continuous-time systems. In case of the mobile
robot (\ref{eq:WMR}), the flatness and the particular flat output
are preserved because of the special triangular structure of the system.

\section{Conclusion}

We have shown that every flat discrete-time system can be decomposed
by state- and input transformations into a subsystem and an endogenous
dynamic feedback. This remarkable feature can be considered as discrete-time
counterpart to the existence of a PAI-form (\ref{eq:PAI-form}) for
flat continuous-time systems, which is closely related to the well-known
ruled-manifold necessary condition. In contrast to the PAI-form or
the ruled-manifold criterion, such a decomposition directly gives
rise to an algorithm which allows to check the flatness of a discrete-time
system in at most $n-1$ steps. If the system is
flat, then the algorithm yields a flat output which only depends on
the state variables. Consequently, every flat discrete-time system
has a flat output which does not depend on the inputs and their forward-shifts.
Compared to the complexity of the flatness problem in the continuous-time
case, these results represent a fundamental simplification.
From a computational point of view, it would nevertheless be desirable
to avoid the coordinate transformations that have to be performed
in each of the steps. Thus, current research is concerned
with the development of a coordinate-independent test for flatness.
More precisely, the idea is to separate the test for flatness from
the calculation of a flat output, similar to the test for static feedback
linearizability. Furthermore, motivated by the existence of flat outputs
which only depend on the state variables, future work will address
the question whether there exist suitable normal forms for flat discrete-time
systems.

\begin{ack}                               
The first author and the third author have been supported by the Austrian Science Fund (FWF) under grant number P 29964 and P 32151.  
\end{ack}

\appendix

\section{Appendix}

The purpose of this section is to illustrate a computationally
efficient construction of projectable linear combinations of the input
vector fields $\partial_{u}$.

In adapted coordinates (\ref{eq:adapted_coordinates}),
the $m$ input vector fields $\partial_{u}$ are of the form 
\begin{align}
v_{1} & =a_{1}^{1}(x^{+},\xi)\partial_{x^{1,+}}+\ldots+a_{1}^{n}(x^{+},\xi)\partial_{x^{n,+}}\nonumber \\
 & \hphantom{=}+b_{1}^{1}(x^{+},\xi)\partial_{\xi^{1}}+\ldots+b_{1}^{m}(x^{+},\xi)\partial_{\xi^{m}}\nonumber \\
 & \vdots\label{eq:basis_general}\\
v_{m} & =a_{m}^{1}(x^{+},\xi)\partial_{x^{1,+}}+\ldots+a_{m}^{n}(x^{+},\xi)\partial_{x^{n,+}}\nonumber \\
 & \hphantom{\phantom{=}}+b_{m}^{1}(x^{+},\xi)\partial_{\xi^{1}}+\ldots+b_{m}^{m}(x^{+},\xi)\partial_{\xi^{m}}\,.\nonumber 
\end{align}
Now we have to check whether there exists a linear combination
\begin{equation}
c^{1}(x^{+},\xi)v_{1}+\ldots+c^{m}(x^{+},\xi)v_{m}\label{eq:linear_combination}
\end{equation}
which is of the form
\[
a^{i}(x^{+})\partial_{x^{i,+}}+b^{j}(x^{+},\xi)\partial_{\xi^{j}}\,,
\]
i.e., projectable. The criterion is that the resulting coefficients
$c^{k}(x^{+},\xi)a_{k}^{i}(x^{+},\xi)$ of the linear combination
(\ref{eq:linear_combination}) in the directions $\partial_{x^{i,+}}$,
$i=1,\ldots,n$ must be independent of $\xi$, i.e., they must satisfy
\begin{equation}
\partial_{\xi^{j}}\left(c^{k}(x^{+},\xi)a_{k}^{i}(x^{+},\xi)\right)=0\label{eq:conditions_general}
\end{equation}
for all $i=1,\ldots,n$ and $j=1,\ldots,m$. To avoid the partial
derivatives of the unknown coefficients $c^{k}(x^{+},\xi)$, it is
beneficial to use a normalized basis
\begin{align}
v_{1} & =\partial_{x^{1,+}}\nonumber \\
 & \hphantom{\phantom{=}}+a_{1}^{m+1}(x^{+},\xi)\partial_{x^{m+1,+}}+\ldots+a_{1}^{n}(x^{+},\xi)\partial_{x^{n,+}}\nonumber \\
 & \hphantom{\phantom{=}}+b_{1}^{1}(x^{+},\xi)\partial_{\xi^{1}}+\ldots+b_{1}^{m}(x^{+},\xi)\partial_{\xi^{m}}\label{eq:basis_normed}\\
 & \vdots\nonumber \\
v_{m} & =\partial_{x^{m,+}}\nonumber \\
 & \hphantom{\phantom{=}}+a_{m}^{m+1}(x^{+},\xi)\partial_{x^{m+1,+}}+\ldots+a_{m}^{n}(x^{+},\xi)\partial_{x^{n,+}}\nonumber \\
 & \hphantom{\phantom{=}}+b_{m}^{1}(x^{+},\xi)\partial_{\xi^{1}}+\ldots+b_{m}^{m}(x^{+},\xi)\partial_{\xi^{m}}\nonumber 
\end{align}
for the distribution spanned by the vector fields (\ref{eq:basis_general}).
Up to a renumbering of the state variables, this can always be achieved
by suitable linear combinations.\footnote{Because of $\mathrm{rank}(\partial_{u}f)=m$, the
matrix formed by the coefficients $a_{k}^{i}$ of (\ref{eq:basis_general})
has full rank $m$.} Because of the $m\times m$ identity matrix in the
coefficients of the normalized basis (\ref{eq:basis_normed}), the
equations (\ref{eq:conditions_general}) with $i=1,\ldots,m$ imply
that all coefficients $c^{1},\ldots,c^{m}$ must be independent of
$\xi$. Consequently, the remaining equations of (\ref{eq:conditions_general})
with $i=m+1,\ldots,n$ simplify to the $(n-m)m$ algebraic equations
\[
c^{k}(x^{+})\partial_{\xi^{j}}a_{k}^{i}(x^{+},\xi)=0\, ,\, i=m+1,\ldots,n\,,\:j=1,\ldots,m\,.
\]

\bibliographystyle{plain}        
\bibliography{Bibliography}           

\begin{thebibliography}{10}

\bibitem{Aranda-BricaireKottaMoog:1996}
E.~Aranda-Bricaire, {\"U}.~Kotta, and C.H. Moog.
\newblock Linearization of discrete-time systems.
\newblock {\em SIAM Journal on Control and Optimization}, 34(6):1999--2023,
  1996.

\bibitem{Aranda-BricaireMoog:2008}
E.~Aranda-Bricaire and C.H. Moog.
\newblock Linearization of discrete-time systems by exogenous dynamic feedback.
\newblock {\em Automatica}, 44(7):1707--1717, 2008.

\bibitem{Boothby:1986}
W.M. Boothby.
\newblock {\em An Introduction to Differentiable Manifolds and Riemannian
  Geometry}.
\newblock Academic Press, Orlando, 2nd edition, 1986.

\bibitem{FliessLevineMartinRouchon:1992}
M.~Fliess, J.~L{\'e}vine, P.~Martin, and P.~Rouchon.
\newblock Sur les syst{\`e}mes non lin{\'e}aires diff{\'e}rentiellement plats.
\newblock {\em Comptes rendus de l'Acad{\'e}mie des sciences. S{\'e}rie I,
  Math{\'e}matique}, 315:619--624, 1992.

\bibitem{FliessLevineMartinRouchon:1995}
M.~Fliess, J.~L{\'e}vine, P.~Martin, and P.~Rouchon.
\newblock Flatness and defect of non-linear systems: introductory theory and
  examples.
\newblock {\em International Journal of Control}, 61(6):1327--1361, 1995.

\bibitem{FliessLevineMartinRouchon:1999}
M.~Fliess, J.~L{\'e}vine, P.~Martin, and P.~Rouchon.
\newblock A {L}ie-{B}{\"a}cklund approach to equivalence and flatness of
  nonlinear systems.
\newblock {\em IEEE Transactions on Automatic Control}, 44(5):922--937, 1999.

\bibitem{Grizzle:1986}
J.W. Grizzle.
\newblock Feedback linearization of discrete-time systems.
\newblock In A.~Bensoussan and J.L. Lions, editors, {\em Analysis and
  Optimization of Systems}, volume~83 of {\em Lecture Notes in Control and
  Information Sciences}, pages 273--281. Springer, Berlin, 1986.

\bibitem{Grizzle:1993}
J.W. Grizzle.
\newblock A linear algebraic framework for the analysis of discrete-time
  nonlinear systems.
\newblock {\em SIAM Journal on Control and Optimization}, 31(4):1026--1044,
  1993.

\bibitem{HuntSu:1981}
L.~Hunt and R.~Su.
\newblock Linear equivalents of nonlinear time varying systems.
\newblock In {\em Proceedings 5th International Symposium on Mathematical
  Theory of Networks and Systems (MTNS)}, pages 119--123, 1981.

\bibitem{Jakubczyk:1987}
B.~Jakubczyk.
\newblock Feedback linearization of discrete-time systems.
\newblock {\em Systems \& Control Letters}, 9(5):411--416, 1987.

\bibitem{JakubczykRespondek:1980}
B.~Jakubczyk and W.~Respondek.
\newblock On linearization of control systems.
\newblock {\em Bull. Acad. Polonaise Sci. Ser. Sci. Math.}, 28:517--522, 1980.

\bibitem{Kaldmae:2016}
A.~Kaldm{\"a}e.
\newblock {\em Advanced Design of Nonlinear Discrete-time and Delayed Systems}.
\newblock PhD thesis, Tallinn University of Technology, 2016.

\bibitem{KaldmaeKotta:2013}
A.~Kaldm{\"a}e and {\"U}.~Kotta.
\newblock On flatness of discrete-time nonlinear systems.
\newblock In {\em Proceedings 9th IFAC Symposium on Nonlinear Control Systems
  (NOLCOS)}, pages 588--593, 2013.

\bibitem{Kolar:2017}
B.~Kolar.
\newblock {\em Contributions to the Differential Geometric Analysis and Control
  of Flat Systems}.
\newblock Shaker Verlag, Aachen, 2017.

\bibitem{KolarKaldmaeSchoberlKottaSchlacher:2016}
B.~Kolar, A.~Kaldm{\"a}e, M.~Sch{\"o}berl, {\"U}.~Kotta, and K.~Schlacher.
\newblock Construction of flat outputs of nonlinear discrete-time systems in a
  geometric and an algebraic framework.
\newblock {\em IFAC-PapersOnLine}, 49(18):796--801, 2016.

\bibitem{KolarSchoberlSchlacher:2016-2}
B.~Kolar, M.~Sch{\"o}berl, and K.~Schlacher.
\newblock A decomposition procedure for the construction of flat outputs of
  discrete-time nonlinear control systems.
\newblock In {\em Proceedings 22nd International Symposium on Mathematical
  Theory of Networks and Systems (MTNS)}, pages 775--782, 2016.

\bibitem{KolarSchoberlSchlacher:2016-3}
B.~Kolar, M.~Sch{\"o}berl, and K.~Schlacher.
\newblock Properties of flat systems with regard to the parameterization of the
  system variables by the flat output.
\newblock {\em IFAC-PapersOnLine}, 49(18):814--819, 2016.

\bibitem{NijmeijervanderSchaft:1990}
H.~Nijmeijer and A.J. van~der Schaft.
\newblock {\em Nonlinear Dynamical Control Systems}.
\newblock Springer, New York, 1990.

\bibitem{Rouchon:1994}
P.~Rouchon.
\newblock Necessary condition and genericity of dynamic feedback linearization.
\newblock {\em Journal of Mathematical Systems, Estimation, and Control},
  4(2):1--14, 1994.

\bibitem{Saunders:1989}
D.J. Saunders.
\newblock {\em The Geometry of Jet Bundles}.
\newblock Cambridge University Press, Cambridge, 1989.

\bibitem{SchlacherSchoberl:2007}
K.~Schlacher and M.~Sch{\"o}berl.
\newblock Construction of flat outputs by reduction and elimination.
\newblock In {\em Proceedings 7th IFAC Symposium on Nonlinear Control Systems
  (NOLCOS)}, pages 666--671, 2007.

\bibitem{SchlacherSchoberl:2013}
K.~Schlacher and M.~Sch{\"o}berl.
\newblock A jet space approach to check {P}faffian systems for flatness.
\newblock In {\em Proceedings 52nd IEEE Conference on Decision and Control
  (CDC)}, pages 2576--2581, 2013.

\bibitem{Schoberl:2014}
M.~Sch{\"o}berl.
\newblock {\em Contributions to the Analysis of Structural Properties of
  Dynamical Systems in Control and Systems Theory - A Geometric Approach}.
\newblock Shaker Verlag, Aachen, 2014.

\bibitem{SchoberlSchlacher:2014}
M.~Sch{\"o}berl and K.~Schlacher.
\newblock On an implicit triangular decomposition of nonlinear control systems
  that are 1-flat - a constructive approach.
\newblock {\em Automatica}, 50:1649--1655, 2014.

\bibitem{Sira-RamirezAgrawal:2004}
H.~Sira-Ramirez and S.K. Agrawal.
\newblock {\em Differentially Flat Systems}.
\newblock Marcel Dekker, New York, 2004.

\end{thebibliography}




\end{document}